\documentclass{amsart}

\usepackage{amsmath,amsfonts,amssymb}
\usepackage{array,delarray}
\usepackage[dvips]{graphicx}
\usepackage{dsfont}
\usepackage{xspace}
\usepackage{color}
\usepackage[sort&compress,square]{natbib}

\newcommand{\ff}{\mathds{F}}
\newcommand{\zz}{\mathds{Z}}

\newcommand{\s}{\mathcal{S}}
\newcommand{\cc}{\mathcal{C}}
\newcommand{\jj}{\mathcal{J}}

\newcommand{\h}{\mathrm{h}}
\newcommand{\p}{\mathrm{p}}
\newcommand{\lcm}{\mathrm{lcm}}

\newcommand{\vv}{\mathbf{v}}
\newcommand{\uu}{\mathbf{u}}

\newcommand{\pp}{\mathfrak{p}}

\newcommand{\dd}{\mathfrak{D}}

\newtheorem{theorem}{Theorem}[section]   % Numbered within each section
\newtheorem{corollary}[theorem]{Corollary}     % Numbered along with thm
\newtheorem{lemma}[theorem]{Lemma}         % Numbered along with thm
\newtheorem{proposition}[theorem]{Proposition}  % Numbered along with thm

\theoremstyle{definition}
\newtheorem{definition}[theorem]{Definition}   % Numbered along with thm

\theoremstyle{remark}
\newtheorem{remark}[theorem]{Remark}        % Numbered along with thm
\newtheorem{example}[theorem]{\emph{Example}}        % Numbered along with thm

\numberwithin{equation}{section}     % Number equations within sections

\def\ds {\displaystyle}

\begin{document}

\title[The Dynamics of Conjunctive and Disjunctive Boolean Networks]{The Dynamics of Conjunctive and Disjunctive Boolean Networks}

\author[Jarrah, Laubenbacher]{Abdul Salam Jarrah, Reinhard Laubenbacher, Alan Veliz-Cuba}
\address[Abdul Salam Jarrah, Reinhard Laubenbacher, Alan Veliz-Cuba] {Virginia Bioinformatics Institute, Virginia Tech, Blacksburg, VA 24061-0477, USA}
\email{\{ajarrah,reinhard,alanavc\}@vbi.vt.edu}

\thanks{This work was supported partially by NSF Grant DMS-0511441.}
\date{\today}


\begin{abstract}
The relationship between the properties of a dynamical system and
the structure of its defining equations has long been studied in
many contexts. Here we study this problem for the class of
conjunctive (resp. disjunctive) Boolean networks, that is, Boolean networks in which
all Boolean functions are constructed with the AND (resp. OR) operator only.
The main results of this paper describe network
dynamics in terms of the structure of the network dependency graph (topology).
For a given such network, all possible limit cycle lengths are
computed and lower and upper bounds for the number of cycles of
each length are given.
In particular, the exact number of fixed points is obtained.
The bounds are in terms of structural
features of the dependency graph and its
partially ordered set of strongly connected components.
For
networks with strongly connected dependency graph, the exact cycle
structure is computed.
\end{abstract}

\maketitle

%%%%%%%%%%%%%%%%%%%%%%%%%%%%%%%%%%%%%%%%%%%%%%%%%%%%%%%%%%%%%%%
\section{Introduction}
The understanding of the relationship between structural features
of dynamical systems and the resulting dynamics is an important
problem that has been studied extensively in the dynamical systems literature.
For example,
work by Golubitsky and Stewart \cite{GS} about coupled cell dynamical
systems given by coupled systems of ODEs attempts to obtain
information about system dynamics from the topology of the graph
of couplings.
Albert and Othemer \cite{AO} used a Boolean network model to show
that the expression patterns of the segment polarity in \emph{Drosophila}
are determined by the topology of its gene regulatory network.
Thomas et al \cite{ThoThiKau}
conjectured that negative feedback loops are necessary for periodic
dynamics whereas positive feedback loops are necessary for
multistationarity. These conjectures have been the subject of many
published articles \cite{Soule,Plahte,Gouze}.
In \cite{SV-CLJ} we demonstrated
that networks with a large number of independent negative feedback loops
tend to have few limit cycles and these cycles are usually long.

In this paper we study the effect of the network topology on the
dynamics of a family of Boolean networks.
Boolean networks in general, and cellular automata in particular,
have long been used to model and simulate a wide range of
phenomena, from logic circuits in engineering and gene regulatory
networks in molecular biology \cite{kauffman69b,Kauff69,Kwon-Cho,DavBor,El, AO} to population dynamics
and the spread of epidemics \cite{Youssef,SKT}.  Especially for large networks,
e.g., many agent-based simulations, it becomes infeasible to
simulate the system extensively in order to obtain information about its
dynamic properties, even if such simulation is possible. In such
instances it becomes important from a practical point of view to
be able to derive information about network dynamics from
structural information. But the problem is of interest in its own
right, in particular in the more general context of time-discrete
dynamical systems over general finite fields. These have been
studied extensively, see, for example, \cite{Vivaldi:94a, BG,MM,
CLP, CJLS}. They have a wide range of applications in engineering
\cite{C,El, LBL, ML1, ML2, MW, WM} and recently in computational
biology \cite{LS,JLSS}.

Let $\ff_2:=\{0,1\}$ be the
Galois field with two elements. We view a Boolean network on $n$
variables as a dynamical system
$$
f = (f_1, \ldots , f_n): \ff_2^n\longrightarrow \ff_2^n.
$$
Here, each of the coordinate functions $f_i:\ff_2^n\longrightarrow
\ff_2$ is a Boolean function which can be written uniquely as a polynomial
where the exponent of each variable in each term is at most one \cite{LN}.
In particular $AND(a,b) = ab$ and $OR(a,b) =a+b+ab$.
We use polynomial forms of Boolean functions throughout this paper.

Two directed graphs are usually assigned
to each such system: The \emph{dependency graph} which encodes the static
relationships among the nodes of the network and the \emph{phase space} which
describe the dynamic behavior of the network.
In this paper we focus on the question of deriving information
about the phase space of a Boolean network from the structure of
its dependency graph. Next we define these graphs.

The \emph{dependency graph} $\dd(f)$
of $f$ has $n$ vertices corresponding to the
Boolean variables $x_1,\ldots ,x_n$ of $f$. There is a directed
edge $i \rightarrow j$ if $x_i$ appears in the function $f_j$.
That is, $\dd(f)$ encodes the variable dependencies in $f$. It is
similar to the coupling graph of Stewart and Golubitsky.

\begin{example}\label{ex:run}
The dependency graph of $f =
(x_2x_3,x_1,x_2,x_3x_4,x_1x_6,x_3x_4x_5)
: \ff_2^6 \longrightarrow \ff_2^6$
is the directed graph in Figure \ref{dep7}.
\end{example}
\begin{figure}[htbp]
\begin{center}
\includegraphics[width=1in]{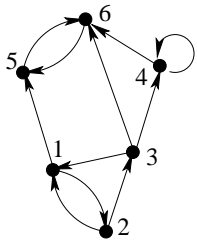}
\caption{The dependency graph of $f$ from Example \ref{ex:run}.}
\label{dep7}
\end{center}
\end{figure}

The dynamics of $f$ is encoded by its \emph{phase space}, denoted
by $\s(f)$. It is the directed graph with vertex set $\ff_2^n$ and
a directed edge from $\uu$ to ${\vv}$ if $f(\uu) = \vv$. For each
$\uu \in \ff_2^n$, the sequence $\{ \uu, f(\uu), f^2(\uu), \dots
\}$ is called the \emph{orbit} of $\uu$. If $\uu = f^t(\uu)$ and
$t$ is the smallest such number, the sequence $\{\uu, f(\uu),
f^2(\uu), \dots, f^{t-1}(\uu)\}$ is called a \emph{limit cycle} of
length $t$ denoted by $\cc_t$, and $\uu$ is called a
\emph{periodic} point of \emph{period} $t$. The point $\uu$ is
called a \emph{fixed point} if $f(\uu)=\uu$. If every limit cycle
is of length 1, the system $f$ is called a \emph{fixed-point}
system. Since $\ff_2^n$ is finite, every orbit must include a
limit cycle. We denote the cycle structure of $f$ in the form of
the generating function
\begin{equation}\label{cyc-struct}
\cc(f) = \sum_{i=1}^\infty C(f)_i \cc_i,
\end{equation}
where $C(f)_i$ denotes the number of cycles $\cc_i$ of length $i$
in the phase space of $f$. Since the phase space of $f$ is finite,
$C(f)_i=0$ for almost all $i$.

The height of $f$, denoted by $\h(f)$, is the least positive
integer $s$  such that $f^s(\uu)$ is a periodic point for all $\uu
\in \ff_2^n$. A \emph{component} of the phase space $\s(f)$
consists of a limit cycle and all orbits of $f$ that contain it.
Hence, the phase space is a disjoint union of components. We
define the \emph{period} of $f$, denoted by $\p(f)$, to be the
least common multiple of the lengths of all limit cycles in the
phase space of $f$.

\begin{example}\label{ex:general}
Let $f : \ff_2^3 \longrightarrow \ff_2^3$ be given by $f(x_1,x_2,x_3) =
(x_2x_3,x_1+x_3,x_1x_2)$. The phase space of $f$ has two components,
containing one fixed point and one limit cycle of length two,
that is $\cc(f) = \cc_1 + \cc_2$, see Figure \ref{fig1}.
It is clear from the phase space that $\h(f)= \p(f) = 2$.
The phase space in Figure \ref{fig1} was generated using DVD \cite{dvd}.
\end{example}
\begin{figure}[htbp]
\begin{center}
\includegraphics[width=2in]{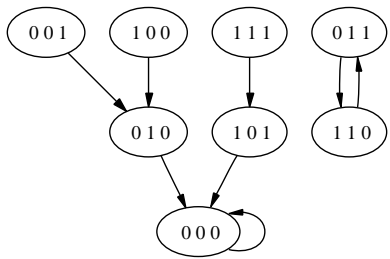}
\caption{The phase space $\s(f)$ of the system $f$ from Example \ref{ex:general}.}
\label{fig1}
\end{center}
\end{figure}

Without exhaustive iteration and just by analyzing $\dd(f)$ what can we say about $\s(f)$?
Namely, what is the period of $f$, the height of $f$, or the
generating function $\cc(f)$? This question is NP-hard in general,
so it is important to limit the class of Boolean networks
considered. Next we list some of the main known results.
\begin{itemize}
\item When all coordinate Boolean functions are the XOR function
(that is, the functions are linear polynomials),
the above questions have been answered completely. In fact the
questions have been answered for linear systems over any Galois field
\cite{El,C,He}. We have
developed and implemented algorithms that answer the questions above, see \cite{JLSV}.
\item For Boolean networks where all coordinate functions are symmetric threshold functions,
it has been shown that all cycles in the phase space are either fixed points
or of length two \cite{GO},
\item For Boolean cellular automata with the majority rule, the number of fixed points was
determined in \cite{AFK}, and
\item In \cite{ADG}, the authors studied
AND-OR networks (Boolean network where each local function is either an OR or AND function)
with directed dependency graphs. Formulae for the maximum number
of fixed points are obtained, and
\item The main result in  \cite{aracena_FB_BRN} is an upper bound for the number of
fixed points in Boolean regulatory networks.
\end{itemize}

In this paper we focus on the class of conjunctive  and disjunctive Boolean
networks, that is, Boolean networks where all of their coordinate functions are
either the AND function or the OR function.
The following represent
main previous attempts to mathematically analyze this class of Boolean networks.
\begin{itemize}
\item A family of this class of networks have been analyzed in \cite{golesHer}.
The authors
studied conjunctive Boolean networks (which they called OR-nets) and disjunctive
Boolean networks (AND-nets) on undirected dependency graph. That is, on graphs where
each edge is  bidirectional and
hence the dependency graph consists of cycles of length two. In particular, their result
that OR-nets have only fixed points and possibly a limit cycle of length two
\cite[Lemma 1]{golesHer} follows directly from our results as we explain in
Remark \ref{lit-rem}.
\item In \cite{barrett-chen}, the authors study a smaller family where each edge is
undirected and each node in the network has a self loop. In particular, they
showed that disjunctive Boolean networks (which they called OR-PDS) have only fixed points
as limit cycles \cite[Theorem 3.3]{barrett-chen}; this follows from our results as we show in
Remark \ref{lit-rem}.
\item Another family consists of the conjunctive Boolean cellular automata
and was analyzed in \cite{BG}. In particular, upper bounds for the
number of cycles are given. Here we study the whole class of
conjunctive Boolean networks and provide lower and upper bounds for the number of their
limit cycles. In particular, we present formulas for the exact number of
fixed points of any conjunctive or disjunctive boolean network, see Equation (\ref{fp-exact}).
\end{itemize}

In this paper we focus only on conjunctive Boolean networks, since
for any disjunctive Boolean network there exists a conjunctive Boolean network
that has exactly the same dynamics after relabeling the 0 and 1 and
hence we have the following theorem.
\begin{theorem}
Let $f=(f_1,\dots,f_n): \ff_2^n \longrightarrow \ff_2^n$ where
$f_i = x_{i_1}\vee \cdots \vee x_{i_{j_i}}$ be any disjunctive Boolean network.
Consider the conjunctive Boolean network $g=(g_1,\dots,g_n):\ff_2^n \longrightarrow \ff_2^n$,
where $g_i = x_{i_1}\wedge \cdots \wedge x_{i_{j_i}}$.
Then the two phase spaces $\s(f)$ and $\s(g)$ are isomorphic as directed graphs.
\end{theorem}
\begin{proof}
Let $\neg:\ff_2^n \longrightarrow \ff_2^n$ be defined by $\neg(x_1,\dots,x_n) = (1+x_1,\dots,1+x_n)$.
Then it is easy to see that $f(x_1,\dots,x_n) = (\neg \circ g \circ \neg)(x_1,\dots,x_n)$.
Thus $\s(f)$ and $\s(g)$ are isomorphic directed graphs.
\end{proof}
\begin{remark}
Let $G$ be a graph on  $n$ vertices such that the in-degree for
each vertex is non-zero ($G$ has no \emph{source vertex}). Then
there is one and only one conjunctive network $f$ on
$n$ nodes such that $\dd(f) = G$. Thus there is a one-to-one
correspondence between the set of all conjunctive Boolean networks
on $n$ nodes where none of the local functions is constant
and the set of directed graphs on $n$ vertices where none of the
vertices is a source.
\end{remark}

This correspondence was used in \cite{CLP} to find the period of a
given Boolean monomial system (conjunctive Boolean network)
and to decide when that system is a
fixed point system as we will recall in the next section. In this
paper we present upper and lower bounds on the number of cycles of
any length in the phase space of any conjunctive Boolean network.
Furthermore, we give upper bounds for the height.

In the next section, we recall some results from graph theory as
well as results about powers of positive matrices that we will use
to obtain upper bounds for the lengths of transients.

%%%%%%%%%%%%%%%%%%%%%%%%%%%%%%%%%%%%%%%%%%%%%%%%%%%%%%%%%%%%%%
%%%%%%%%%%%%%%%%%%%%%%%%%%%%%%%%%%%%%%%%%%%%%%%%%%%%%%%%%%%
\section{The Relationship Between Dependency Graph and Dynamics}

Let $f: \ff_2^n \longrightarrow \ff_2^n$ be a conjunctive Boolean
network, $G = \dd(f)$ and $A$ the adjacency matrix of $G$. We
will assume here and in the remainder of the paper that none of
the Boolean coordinate functions of $f$ are constant, that is, all
vertices of $G$ have positive in-degree.

\subsection{The Adjacency Matrix} \label{adj-mat-sec}
Define the following
relation on the vertices of $G$: $a \sim b$ if and only if there is
a directed path from $a$ to $b$ and a directed path from $b$ to $a$.
It is easy to check that $\sim$ is an equivalence relation. Suppose
there are $t$ equivalence classes $V_1,\dots, V_t$. For each
equivalence class $V_i$, the subgraph $G_i$ with the vertex set
$V_i$ is called a \emph{strongly connected component} of $G$.
The graph $G$ is called \emph{strongly connected} if $G$ has a unique
strongly connected component.

There exists a
permutation matrix $P$ that permutes the rows and columns of $A$
such that
\begin{equation}\label{frob}
PAP^{-1} =
\begin{bmatrix}
  A_{1} & A_{12} & \cdots & A_{1t} \\
     0  & A_2 & \cdots & A_{2t} \\
 \vdots & \vdots & \ddots & \vdots \\
    0   & 0      & \cdots & A_{t} \\
\end{bmatrix}
\end{equation}
where $A_i$ is the adjacency matrix of the component $G_i$, and
$A_{ij}$ represents the edges from the component $G_i$ to $G_j$,
see \cite[Theorem 3.2.4]{CombMath}.
The form in (\ref{frob}) is called the \emph{Frobenius Normal Form} of $A$.

\begin{remark}
The effect of the matrix $P$ on the dependency graph is the
relabeling of the vertices of $G$ such that the diagonal blocks
correspond to strongly connected components of $G$. In Example
\ref{ex:run} above, the adjacency matrix of the dependency graph
is in the normal form.
\end{remark}

%%%%%%%%%%%
% Example cont.
\vspace{.2cm} \noindent {\bf \emph{Example}} \ref{ex:run} (Cont.). The
dependency graph $G$ in Figure \ref{dep7} has 3 strongly connected
components and their vertex sets are: $V_1 = \{1,2,3\}, V_2 =
\{4\}$, and $V_3 = \{5,6\}$, see Figure \ref{disjoint-comp} (left).

\vspace{.2cm} \noindent For any non-empty strongly connected
component $G_i$, let $h_i$ be
the conjunctive Boolean network with dependency graph $\dd(h_i) =
G_i$.
Let $h: \ff_2^n \longrightarrow \ff_2^n$ be the conjunctive
Boolean network defined by $h = (h_1,\dots,h_t)$. That is, the
dependency graph of $h$ is the disjoint union of the strongly
connected graphs $G_1,\ldots ,G_t$.

%%%%%%%%%%%
% Example cont.
\vspace{.2cm} \noindent{\bf \emph{Example}} \ref{ex:run} (Cont.). The
conjunctive Boolean network $h$ corresponding to the disjoint
union is $h:
\ff_2^6\longrightarrow \ff_2^6$ and given by $h(x_1,\dots,x_6)
=(h_1(x_1,x_2,x_3),h_2(x_4),h_3(x_5,x_6))$ where
$h_1(x_1,x_2,x_3)= (x_2x_3,x_1,x_2)$, \,
$h_2(x_4) = x_4$, and
$h_3(x_5,x_6)= (x_6,x_5)$.

\noindent
Now define the following relation among the strongly connected
components $G_1,\dots,G_t$ of the dependency graph $\dd(f)$ of the network $f$.
\begin{equation}
G_i \preceq G_j \mbox{ if there is at least one edge from a vertex in } G_i \mbox{ to a vertex in } G_j.
\end{equation}
Since $G_1,\dots,G_t$ are the strongly connected components of
$\dd(f)$, the set of strongly components with the relation
$\preceq$ is a partially ordered set $\mathcal{P}$. In this paper,
we relate the dynamics of $f$ to the dynamics of its strongly
connected components and their poset $\mathcal{P}$.

%%%%%%%%%%%
% Example cont.
\vspace{.2cm}
\noindent {\bf \emph{Example}} \ref{ex:run} (Cont.). The poset of the strongly connected components of $f$
is in Figure \ref{disjoint-comp}(right).
\begin{figure}[htbp]
\begin{center}
\includegraphics[width=3in]{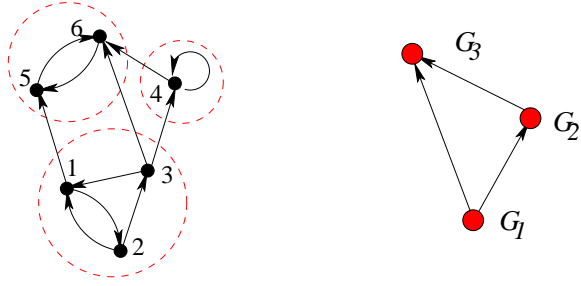}
\caption{The strongly connected components of $f$ (left) and their poset (right).}
\label{disjoint-comp}
\end{center}
\end{figure}

\vspace{.2cm}\noindent For any non-negative matrix $A$, the
sequence $\{A,A^2,\dots\}$ has been studied extensively, see, for
example, \cite{CombMath,sachkov}. Next we use the Boolean
operators AND and OR to examine the sequence of powers of $A$ and
infer results about the conjunctive Boolean network that
corresponds to an adjacency matrix $A$.

%%%%%%%%%%%%%%%%%%%%%%%%%%%%%%%%%%%%%%%%%%%%%%%%%%%%%%%%%%%%%%%%%%%
\subsection{Powers of Boolean Matrices}

Let $A,B$ be $n \times n$ Boolean matrices (all entries are either
$0$ or $1$). Define $A \otimes B$ such that $(A\otimes B)_{ij} =
\bigvee_{k=1}^n (A_{ik} \wedge B_{kj})$, where $\vee$ (resp.
$\wedge$) is the Boolean OR (resp. AND) operator.
\begin{proposition}\label{power-iterate}
Let $A, B$ be as above and let $f,g : \ff_2^n \longrightarrow
\ff_2^n$ be the two conjunctive Boolean networks that correspond
to the adjacency matrices $A$ and $B$, respectively. That is, $f_i
= x_1^{a_{i1}}x_2^{a_{i2}} \cdots x_n^{a_{in}}$ and $g_i =
x_1^{b_{i1}}x_2^{b_{i2}} \cdots x_n^{b_{in}}$ for all $i$. Then
the adjacency matrix corresponding to $f \circ g$ is $A \otimes
B$.
\end{proposition}
\begin{proof}
It is easy to see that, for all $i$,
\begin{eqnarray*}
f_i(g_1,\dots,g_n) &=& g_1^{a_{i1}} \cdots g_n^{a_{in}} \\
                    &=& (x_1^{b_{11}}x_2^{b_{12}} \cdots x_n^{b_{1n}})^{a_{i1}} \cdots (x_1^{b_{n1}}x_2^{b_{n2}} \cdots x_n^{b_{nn}})^{a_{in}} \\
                    &=& x_1^{a_{i1}b_{11}+\cdots+a_{in}b_{n1}} \cdots x_n^{a_{i1}b_{1n}+\cdots+a_{in}b_{nn}} \\
                    &=& x_1^{\sum_{j=1}^n a_{ij}b_{j1}} \cdots x_n^{\sum_{j=1}^n a_{ij}b_{jn}}
\end{eqnarray*}
Since we are working over $\ff_2$, for all $1 \leq k \leq n$, we
have $x_k^q = x_k$ for all positive integers $q$. Thus
\begin{eqnarray*}
x_k \mbox{ divides }  f_i(g_1,\dots,g_n)  &\iff& \sum_{j=1}^n a_{ij}b_{jk} \geq 1 \\
                                        &\iff& a_{ij_0}b_{j_0k} = 1 \mbox{ for some } j_0 \\
                                        &\iff& a_{ij_0} \wedge b_{j_0k} = 1 \mbox{ for some } j_0 \\
                                        &\iff& \bigvee_{j=1}^n a_{ij} \wedge b_{jk} = 1 \\
                                        &\iff& (A\otimes B)_{ik} = 1.
\end{eqnarray*}
Thus the matrix $(A\otimes B)$ is the adjacency matrix of $f \circ g$.
\end{proof}

\noindent Throughout this paper, we use $A^s$ to denote
$\underbrace{A \otimes \cdots \otimes A}_{s \ times}$.

\begin{corollary}
Let $f$ be a conjunctive Boolean network and let $A$ be the
adjacency matrix of its dependency graph $\dd(f)$. Then $A^s$ is
the adjacency matrix of $f^s$.
\end{corollary}

%%%%%%%%%%%%%%%%%%%%%%%%%%%%%%%%%%%%%%%%%%%%%%%%%%%%%%%%%%%%%%%%%%%%
\subsection{The Loop Number}

An invariant of a strongly connected graph, called the loop
number, was defined in \cite{CLP}. We generalize this definition
to any directed graph.
\begin{definition}
The \emph{loop number} of a strongly connected graph is the greatest common divisor of the
lengths of its simple (no repeated vertices) directed cycles.
We define the loop number of a trivial strongly connected graph to be 0.
The loop number of any directed graph $G$ is the least
common multiple of the loop numbers of its non-trivial strongly connected components.
\end{definition}
\begin{remark}
Let $G$ be a directed graph and $A$ its adjacency matrix.
\begin{enumerate}
\item The loop number of $G$ is the same as
    the \emph{index of cyclicity} of $G$  as in \cite{deschut00} and the \emph{index of imprimitivity}
    of $A$ as in \cite{CombMath,nonnegative:1994}.
\item The loop number can be computed in  polynomial time,
    see \cite{CLP} for an algorithm.
\item If the loop number of $G$ is 1, the adjacency matrix $A$ of $G$ is called
      \emph{primitive}, see \cite{CombMath}.
\end{enumerate}
\end{remark}

\vspace{.2cm}
\noindent
{\bf \emph{Example}} \ref{ex:run} (Cont.).
The loop numbers of $V_1,V_2,V_3$ are 1,1,2, respectively. See Figure \ref{disjoint-comp}(left). In particular,
the loop number of $G$ is 2.

\begin{definition}
The exponent of an irreducible matrix $A$ with loop number $c$ is
the least positive integer $k$ such that $A^{k+c} = A^k$.
\end{definition}

The following lemma follows from Proposition \ref{power-iterate} and the definition above.
\begin{lemma}\label{loop-period}
Let $f$ be a conjunctive Boolean network and let $A$ be the adjacency matrix of
its dependency graph $\dd(f)$. Suppose the loop number of $A$ is $c$ and its exponent
is $k$. Then $\h(f) =k$ and $\p(f)$ divides $c$. In particular, $C(f)_i = 0$ for every $i \nmid c$
and hence Equation (\ref{cyc-struct}) becomes
\begin{equation}\label{cyc-struct-1}
\cc(f) = \sum_{i | c} C(f)_i \cc_i,
\end{equation}
\end{lemma}

\vspace{.2cm}
\noindent {\bf \emph{Example}} \ref{ex:run} (Cont.).
The phase space $\s(f)$ has $2^6$ vertices, its cycle
structure is $\cc(f) = 4\cc_1+1\cc_2$.
In particular, the period of $f$ is 2 which is the same as the loop number of
its dependency graph $G$.

\begin{remark}\label{lit-rem}
It is clear that if $x_i$ appears in $f_i$ for all $i$, then
the loop number of $\dd(f)$ is one and hence $\s(f)$ has only fixed points
as limit cycles, this was shown in \cite[Theorem 3.3]{barrett-chen}.
Also if each edge in the dependency graph is
undirected, then the dependency graph is made up of simple cycles of length 2 and
hence the loop number of $\dd(f)$ is either two or one. Thus $\s(f)$
has only fixed points and possibly cycles of length two, which was shown in \cite[Lemma 1]{golesHer}.
\end{remark}

For the case when $\dd(f)$ is strongly connected $\p(f) = c$ as was shown in \cite[Theorem 4.13]{CLP};
in particular, $\s(f)$ has a simple cycle of length $l$ if and only if $l$ divides $c$.
In general, however, this is not the case.

\begin{example}\label{2-3}
Consider the Boolean network $f=(x_2,x_1,x_2x_5,x_3,x_4):\ff_2^5 \longrightarrow \ff_2^5$.
The graph $\dd(f)$ has two strongly connected components with loop numbers 2 and 3 respectively, and hence
the loop number of $\dd(f)$ is 6. However,
it is easy to see, using DVD \cite{dvd},
that $\cc(f) = 3\cc_1+1\cc_2+2\cc_3$, and hence $f$ has no cycle of length 6.
\end{example}

The exponent has been studied extensively and
upper bounds are known, \cite[Theorem 3.11]{deschut00} presents
an upper bound for the exponent of any irreducible matrix. Using
the lemma above, we rewrite this upper bound for the height of any
Boolean monomial system with a strongly connected dependency
graph.

\begin{theorem}
Let $f$ be a conjunctive Boolean network with strongly connected
dependency graph $\dd(f)$, and suppose the loop number of $\dd(f)$
is $c$. Then
\begin{equation*}
\h(f) \leq \left\{%
\begin{array}{ll}
    (n-1)^2+1, & \hbox{ if \,}c = 1;\\
    \max\{n-1,\dfrac{n^2-1}{2}+\dfrac{n^2}{c}-3n+2c\}, & \hbox{ if \,} c > 1.\\
\end{array}%
\right.
\end{equation*}
\end{theorem}

\begin{proof}
The proof follows from Proposition \ref{power-iterate} above and
\cite[Theorem 3.11]{deschut00}.
\end{proof}

Furthermore, \cite[Theorem 3.20]{deschut00} also implies an upper
bound for the height of any conjunctive Boolean network.
\begin{theorem}
Let $f$ be any conjunctive Boolean network on $n$ nodes. Then
$\h(f) \leq 2n^2-3n+2$.
\end{theorem}

\begin{proof}
The proof follows from Proposition \ref{power-iterate} above and
\cite[Theorem 3.20]{deschut00}.
\end{proof}

We close this section with a classical theorem about positive
powers of Boolean matrices. This theorem has the remarkable
corollary that almost all conjunctive Boolean networks have a
strongly connected dependency graph and, furthermore, have only
fixed points as limit cycles.
\begin{theorem}\cite[Theorem 3.5.11]{CombMath} \label{Th:strongly_connected}
Let $N(n)$ be the number of conjunctive Boolean networks on $n$
nodes with strongly connected dependency graph and loop number 1.
Then
\begin{displaymath}
\lim_{n\rightarrow \infty} \dfrac{N(n)}{2^{n^2}} = 1.
\end{displaymath}
In particular, since there are $2^{n^2}$ different conjunctive
Boolean networks on $n$ nodes, as $n \rightarrow \infty$, almost all
conjunctive Boolean networks on $n$ nodes have a strongly connected
dependency graph and have only fixed points as limit cycles.
\end{theorem}
Although there is a wealth of information about powers of
non-negative matrices such as the transient length or possible
cycle length, very little seems to be known about the number of
cycles of a given length in the phase space and that is the main
goal of this paper. In the next section we give a complete answer
to this problem for conjunctive Boolean networks with strongly
connected dependency graph. For this class of systems, we find the
number of cycles of any possible length in the phase space.

%%%%%%%%%%%%%%%%%%%%%%%%%%%%%%%%%%%%%%%%%%%%%%%%%%%%%%%%%%%%%%%%%%%%%%%%%%%%%%%
\section{Networks with Strongly Connected Dependency Graph}
\label{stronglyconnected}

In this section we give an exact formula for the cycle structure
of conjunctive Boolean networks with strongly connected dependency
graphs. Since "almost all" conjunctive Boolean networks have this
property by Theorem \ref{Th:strongly_connected}, one may consider
this result as giving a complete answer for conjunctive Boolean
networks in the limit. However, in the next section we will also
consider networks with general dependency graphs and give upper
and lower bounds for the cycle structure.

Before deriving the desired state space results we prove some
needed facts about general finite dynamical systems.

\begin{lemma}
Let $f : X \longrightarrow X$ be a finite dynamical system, with
$X$ a finite set, and let $\uu  \in X$ be a periodic point of
period $t$.
\begin{enumerate}
\item If $f^s(\uu) = \uu$, then $t$ divides $s$. \label{item1}
\item The period of $f^j(\uu)$ is $t$, for any $ j \geq 1$.
\label{item2} \item If $f^s(\uu) = f^j(\uu)$, then $t$ divides
$s-j$. \label{item3}
\end{enumerate}
\end{lemma}
\begin{proof}
Since $t$ is the period of $\uu$ and $f^s(\uu) = \uu$, by
definition $s \geq t$. Thus $s = qt+r$, where $0 \leq r < t$. Now
$\uu = f^s(\uu) = f^{r}(f^{qt}(\uu))= f^r(\uu)$. Since $r < t$ and
$t$ is the period of $\uu$, $r = 0$ and hence $t$ divides $s$.
That proves (\ref{item1}). The proof of (\ref{item2}) is
straightforward, since $f^t(f^j(\uu)) = f^j(f^t(\uu)) = f^j(\uu)$.
Now we prove (\ref{item3}). Suppose $ s \geq j$. Since $f^s(\uu) =
f^j(\uu)$, we have $f^{s-j}(f^j(\uu)) = f^s(\uu) = f^j(\uu)$.
Thus, by (\ref{item1}), the period of $f^j(\uu)$ divides $s-j$.
But the period of $f^j(\uu)$ is $t$, by (\ref{item2}).
\end{proof}

\begin{lemma}
Let $f: X \longrightarrow X$ be a finite dynamical system. Then,
$\p(f) = c$ and $\h(f) = d$ if and only if $c$ and $d$ are the
least positive integers such that $f^{m+c}(\uu) = f^m(\uu)$ for
all $m \geq d$ and $\uu \in X$.
\end{lemma}
\begin{proof}
Suppose that $\p(f) = c$ and $\h(f) = d$. Then for all $\uu \in X$
and $m \geq d$, $f^m(\uu)$ is a periodic point and hence its
period divides $c$. Thus $f^{m+c}(\uu) = f^c(f^m(\uu)) =
f^m(\uu)$.

Now suppose $c$ and $d$ are the least positive integers such that
$f^{m+c}(\uu) = f^m(\uu)$ for all $m \geq d$ and $\uu \in X$. We
want to show that $\p(f) = c$ and $\h(f) = d$. It is clear that
$f^d(\uu)$ is periodic for all $\uu \in X$ and $d$ is the least
such positive number. Thus the height of $f$ is $d$. Also, the
period $f$ is $c$, since $c$ is the least positive integer such
that $f^c(\uu)  = \uu$ for any periodic point $\uu \in X$.
\end{proof}

\begin{theorem}\label{cross-prod}
Let $f: X \longrightarrow X$ and $g : Y \longrightarrow Y$ be two
finite dynamical systems. Define the system $h : X \times Y
\longrightarrow X \times Y$ by $h(\uu,\uu') = (f(\uu),g(\uu'))$.
Then  $\s(h) = \s(f) \times \s(g)$.
\end{theorem}

\begin{proof}
This follows from the fact that $h(\uu,\uu') = (\vv,\vv')$ if and
only if $f(\uu)= \vv$  and $g(\uu') =\vv'$ for all $\uu \in X$ and
$\uu' \in Y$.
\end{proof}

\begin{corollary}\label{prod:period}
Let $f$, $g$ and $h$ be as in Theorem \ref{cross-prod}. Then the
period of $h$ is $\p(h) = \lcm\{\p(f),\p(g)\}$ and its height is
$\h(h) = \max\{\h(f),\h(g)\}$.
\end{corollary}

\begin{proof}
It is easy to see that a set $A \subset X \times Y$ is a cycle in
$\s(h)$ if and only if $A_X$ (resp. $A_Y$) is a cycle in $\s(f)$
(resp. $\s(g)$), where $A_X := \{\uu \in X \, : (\uu,\vv) \in A
\textit{ for some } \vv \in Y\}$. Furthermore, it is clear that
the length of the cycle $A$ is the least common multiple of the
lengths of $A_X$ and $A_Y$. Thus $\p(h) = \lcm\{\p(f),\p(g)\}$.
The proof of $\h(h) = \max\{\h(f),\h(g)\}$ follows from the
definition of height.
\end{proof}

Recall Equation (\ref{cyc-struct}), in particular, that $C(f)_m$
 is the number of cycles of length $m$ in the phase
space of $f$, the following corollary follows directly from
Theorem \ref{cross-prod}.

\begin{corollary}\label{cross-prod-cycles}
Let $f$, $g$ and $h$ be as above. Then, the cycle structure of $h$
is $\cc(h) = \cc(f)\cc(g) := \sum_{m | \p(h)} C(h)_m \cc_m$, where
\begin{equation}\label{cross-prod-formula}
C(h)_m = \sum_{\substack{s | p(f)\\ t | p(g)\\ \lcm\{s,t\} = m}}
\gcd\{s,t\}C(f)_s C(g)_t.
\end{equation}
That is, the generating function for the cycle structure of $h$ is
the product of the generating functions of $f$ and $g$, where
$\cc_s \cdot \cc_t  = \gcd\{s,t\} \cc_{\lcm\{s,t\}}$.
\end{corollary}

\vspace{.2cm}\noindent
Let $f:\ff_2^n \longrightarrow \ff_2^n$ be a conjunctive Boolean
network. Assume that the dependency graph $\dd(f)$ of $f$ is
strongly connected with loop number $c$. For any divisor $k$ of
$c$, it is well-known that the set of vertices of $\dd(f)$ can be
partitioned into $c$ non-empty sets $W_1,\dots,W_k$ such that each
edge of $\dd(f)$ is an edge from a vertex in $W_i$ to a vertex in
$W_{i+1}$ for some $i$ with $1 \leq i \leq k$ and $W_{k+1} = W_1$.
For a proof of this fact see \cite[Lemma 3.4.1(iii)]{CombMath} or
\cite[Lemma 4.7]{CLP}.

By definition the length of any cycle in the phase space $\s(f)$
of $f$ divides c, the period of $f$. For any positive integers $p,
k$ that divide $c$, let $A(p)$ be the set of periodic points of
period $p$ and let $\ds D(k) := \bigcup_{p|k} A(p)$.
\begin{lemma}
 The cardinality of the set $D(k)$ is $|D(k)|=2^k$.
\end{lemma}

\begin{proof}
Let $\Phi:\ff_2^k \longrightarrow D(k)$ be defined by
\begin{displaymath}
\Phi(x_1,\dots,x_k)=
(\underbrace{x_1,\dots,x_1}_{s_1 \ times}, \underbrace{x_2,\dots,x_2}_{s_2 \ times},\dots, \underbrace{x_k,\dots,x_k}_{s_k \ times}),
\end{displaymath}
where, without loss of generality,  $W_i = \{v_{i,1},\dots,
v_{i,s_{i}}\}$ for all  $1 \leq i \leq k$. We will show that
$\Phi$ is a bijection. First we show that $\Phi$ is well-defined.
Let $z=\Phi(x_1,\dots, x_k)$, by \cite[Theorem 4.10]{CLP}, there
exists a positive integer $m$ such that
\begin{displaymath}
f^{mk+j}(z)= (\underbrace{x_{j+1},\dots,x_{j+1}}_{s_1 \ times}, \underbrace{x_{j+2},\dots,x_{j+2}}_{s_2 \ times}, \dots,  \underbrace{x_j,\dots,x_j}_{s_k \ times}).
\end{displaymath}
In particular, $f^k(z) = f^{mk+k}(z) = z$. Thus $z =
\Phi(x_1,\dots,x_k) \in D(k)$. Since it is clear that $\Phi$ is
one-to one, it is left to show that $\Phi$ is onto.

Let $z\in D(k)$ which means that $f^{k}(z)=z$. It is enough to
show that $z_{i,h} = z_{i,g}$ for all $1 \leq h,g \leq s_{i_k}$
and $1 \leq i \leq k$. Suppose not. Without loss of generality,
let $z_{i,h} = 1 $ and $z_{i,g} = 0$. Since $k$ divides $c$,
$f^c(z) = z$. But there is a path from $v_{i,h}$ and $v_{i,g}$ of
length divides $c$, contradiction. Hence $\Phi$ is a bijection and
$|D(k)|=2^k$.
\end{proof}

\begin{corollary}
If $p$ is a prime number and $p^k$ divides $c$ for some $k \geq 1$, then
\begin{displaymath}
|A(p^k)| = 2^{p^k} - 2^{p^{k-1}}.
\end{displaymath}
\end{corollary}

\begin{proof}
It is clear that $|D(1)| = 2$. Namely, $(0,\dots,0)$ and $(1,\dots,1)$ are the only two fixed points.
Now if $p$ is prime and $k \geq 1$,  then the proof follows from the fact that
$D(p^k) = D(p^{k-1}) \biguplus A(p^k)$, where $\biguplus$ is the disjoint union.
\end{proof}

The following theorem gives the exact number of periodic points of
any possible length.
\begin{theorem}
Let $f$ be a conjunctive Boolean network whose dependency graph is
strongly connected and has loop number $c$.  If $c=1$, then $f$
has the two fixed points $(0,0,\ldots ,0)$ and $(1,1,\ldots ,1)$
and no other limit cycles of any length.  If $c>1$ and $m$ is a
divisor of $c$, then
\begin{equation}
|A(m)|= \sum_{i_1 = 0}^1 \cdots \sum_{i_r=0}^{1} (-1)^{i_1+i_2+\cdots+i_r} 2^{p_1^{k_1-i_1}p_2^{k_2-i_2} \dots p_r^{k_r-i_r}}, \label{Am}
\end{equation}
where
$m= \prod_{i=1}^r p_i^{k_i}$ is the prime factorization of $m$, that is
$p_1,\dots,p_r$ are distinct primes and $k_i\geq 1$ for all $i$.
\end{theorem}

\begin{proof}
The statement for $c=1$ is part of the previous corollary.
Now suppose that $c>1$. For $1 \leq j \leq r$, let $m_j =
p_j^{k_j-1}\prod_{i=1, i\neq j}^r p_i^{k_i}$. Then $D(m) = A(m)
\biguplus (\bigcup_{j=1}^r D(m_j)) $, where $\biguplus$ is a
disjoint union, in particular,
\begin{displaymath}
A(m) = D(m) \setminus \bigcup_{j=1}^r  D(m_j).
\end{displaymath}
The formula \ref{Am} follows from the inclusion-exclusion principle and the disjoint union above.
\end{proof}

\begin{corollary}
If $m$ divides $c$, then the number of cycles of length $m$ in the phase space of $f$ is
$C(f)_m = \dfrac{|A(m)|}{m}$. Hence the cycle structure of $f$ is
\begin{equation*}
\cc(f) = \sum_{m \mid c} \dfrac{|A(m)|}{m} \cc_m.
\end{equation*}
\end{corollary}

\begin{remark}
Notice that the cycle structure of $f$ depends on its loop
number only. In particular, a conjunctive Boolean network with
loop number 1 on a strongly connected dependency graph only
has as limit cycles the two fixed points $(0,0,\ldots ,0)$ and $(1,1,\ldots ,1)$, regardless
of how many vertices its dependency graph has.
\end{remark}

%%%%%%%%%%%%%%%%%%%%%%%%%%%%%%%%%%%%%%%%%%%%%%%%%%%%%%%%%%%%%%%%%%%
\section{Networks with general dependency graph}

Let $f : \ff_2^n \longrightarrow \ff_2^n$ be a conjunctive Boolean
network with dependency graph $\dd(f)$. Without loss of
generality, the adjacency matrix of $\dd(f)$ is in the form~(\ref{frob}).
Let $G_1,\dots,G_t$ be the strongly connected
components of $\dd(f)$ corresponding to the matrices
$A_1,\dots,A_t$, respectively. Furthermore, suppose that none of
the $G_i$ is trivial (i.e., $A_i$ is the zero matrix). For $1 \leq
i \leq t$, let $h_i$ be the conjunctive Boolean network that has
$G_i$ as its dependency graph and suppose that the loop number of
$h_i$ is $l_i$. In particular, the loop number of $f$ is $l :=
\lcm\{l_1,\dots,l_t\}$.

In the remainder of the paper, we present lower and upper bounds
for the number of cycles of a given length. We use the strongly
connected components of the dependency graph and their poset to
infer the cycle structure of $f$.

Let $G_1$ and $G_2$ be two strongly connected components in
$\dd(f)$ and suppose $G_1 \preceq G_2$. Furthermore, assume the vertex set of
$G_1$ (resp. $G_2$) is $\{x_{i_1},\dots, x_{i_s}\}$ (resp.
$\{x_{j_1},\dots, x_{j_t}\}$). Without loss of generality, let
$x_{i_1} \longrightarrow x_{j_1}$ be a directed edge in $\dd(f)$
and let $\dd'$ be the graph $\dd(f)$ after deleting the edge
$(x_{i_1},x_{j_1})$. Let $g$ be the conjunctive Boolean network
such that $\dd(g) = \dd'$.

\begin{theorem}\label{theorem:f-h}
Any cycle in the phase space of $f$ is a cycle in the
phase space of $g$. In particular $\cc(f) \leq \cc(g)$ componentwise.
\end{theorem}
\begin{proof}
Let $\mathcal{C} := \{\uu,f(\uu),\dots,f^{m-1}(\uu)\}$ be a cycle
of length $m$ in $\s(f)$. To show that $\mathcal{C}$ is a cycle in
$\s(g)$, it is enough to show that, whenever $x_{i_1}$-value in
$\uu$ is $0$, there exists $x_{j_w} \in G_2$, such that there is
an edge from $x_{j_w}$ to $x_{j_1}$ and the $x_{j_w}$-value in
$\uu$ is $0$. Thus, the value of $x_{j_1}$ is determined already
by the value of $x_{j_w}$ and the edge $(x_{i-1}, x_{j_1})$ does
not make a difference here and hence $\cc$ is a cycle in $\s(g)$.

Suppose the loop number of $G_1$ (resp. $G_2$) is $a$ (resp. $b$).
Now, any path from $x_{i_1}$ (resp. $y_{j_1}$) to itself is of length $pa$
(resp. $qb$) where $q, p \geq T$ and $T$ is large enough, see
\cite[Corollary 4.4]{CLP}. Thus there is a path from $x_{i_1}$ to
$x_{j_1}$ of length $qa+1$ for any $q \geq T$. Also, there is a
directed path from $x_{j_1}$ to $x_{j_w}$ of length $qb-1$ for any $q \geq
T$. This implies the existence of an edge from $x_{i_1}$ to $x_{j_w}$ of
length $q(a+b)$ for all $q \geq T$.

Now $\uu = f^{m}(\uu) = f^{mk}(\uu)$, for all $k \geq 1$.
Choose $q,k$ large enough such that $q(a+b) = km$.
Then
the value of $x_{j_w}$ in $f^{q(a+b)}(\uu)$ is equal to zero,
since there is a path from $x_{i_1}$ to $x_{j_w}$ of length
$q(a+b)$ and the value of $x_{i_1}$ is zero.
Therefore, the value of $x_{j_w}$ in $\uu$ is zero, since
$\uu = f^{mk}(\uu)= f^{q(a+b)}(\uu)$.
\end{proof}

\begin{corollary}\label{upperbound-g}
Let $f$ and $\dd(f)$ be as above. For any two strongly connected
components in $\dd(f)$ that are connected, drop all but one of the
edges. Let $\dd'$ be the new graph and let $g$ be the conjunctive
Boolean network such that $\dd(g) = \dd'$. Then any cycle in
$\s(f)$ is a cycle in $\s(g)$. In particular $\cc(f) \leq
\cc(g)$ componentwise.
\end{corollary}

Notice that there are many different possible $g$ one can get and
each one of them has the same poset as $f$ and provides an upper
bound for the cycle structure of $f$. However, we do not have a
polynomial formula for the cycle structure of $g$.  When we delete
all edges between any two strongly connected components we get an
easy formula for the cycle structure for the corresponding system
as we describe below. In Section \ref{upper-bound-sec} we will
present an improved upper bound for the cycle structure of $f$.

Let $h :\ff_2^n \longrightarrow \ff_2^n$ be the conjunctive
Boolean network with the disjoint union of $G_1, \dots, G_t$ as
its dependency graph. That is, $h = (h_1,\dots,h_t)$. For $h$,
there are no edges between any two strongly connected components
of the dependency graph of the network. Its cycle structure can be
completely determined from the cycles structures of the $h_i$
alone.

\begin{theorem}\label{disjoint}
Let $\cc(h_i) = \sum_{j|l_i} a_{i,j} \cc_j$ be the cycle structure of $h_i$. Then
the cycle structure of $h$ is $\cc(h)= \prod_{i=1}^t \cc(h_i)$ and the number
of cycles of length $m$ (where $m|l$) in the phase space of $h$ is
\begin{equation}\label{C_m:disjoint}
C(h)_m = \sum_{\substack{j_i | l_i \\ \lcm\{j_1,\dots,j_t\}=m}} \dfrac{j_1 \cdots j_t}{m} \prod_{i=1}^t a_{i,j_i}.
\end{equation}
\end{theorem}

\begin{proof}
This follows directly from Corollary \ref{cross-prod-cycles}.
\end{proof}

\begin{corollary}\label{upperbound}
Let $f$ and $h$ be as above. The number of cycles
of any length in the phase space of $f$ is less than or equal to
the number of cycles of that length in the phase space of $h$.
That is $\cc(f) \leq \cc(h)$ componentwise.
\end{corollary}

\vspace{.2cm}
\noindent
{\bf \emph{Example}} \ref{2-3} (Cont.).\label{2:3}
Here $h=(h_1,h_2)$, where $h_1(x_1,x_2) = (x_2,x_1)$ and
$h_2(x_3,x_4,x_5) = (x_5,x_3,x_4)$. It is easy to check that
$\cc(h_1) =2\cc_1+1\cc_2$, and
$\cc(h_2) =2\cc_1+2\cc_3$. By Theorem \ref{disjoint},
$\cc(h) = 4\cc_1+2\cc_2+4\cc_3+2\cc_6$. In particular, $\cc(f) \leq \cc(h)$.

%%%%%%%%%%%%%%%%%%%%%%%%%%%%%%%%%%%%%%%%%%%%%%%%%%%%%%%%%%%%%%%%
\section{Bounds on the cycle structure}

Let $f$, $h$, $G_1,\dots,G_t$ be as above and let $l_i$ be the
loop number of $G_i$. Furthermore, let $\mathcal{P}$ be the poset
of the strongly connected components. Let $\Omega$ be the set of
all maximal antichains in $\mathcal{P}$. For $J \subseteq [t] =
\{1,\ldots ,t\}$, let $x_J := \prod_{j \in J} x_j$ and let
$\overline{J} := [t]\setminus J$.

\begin{definition}
Let $\mathcal{A}$ be a subset of the set of limit cycles of $h$
and let $s_i$ be the number of limit cycles of length $i$ in
$\mathcal{A}$. We define and denote  $\| \mathcal{A}\|$ by
\[
\| \mathcal{A}\| :=\sum_{i} {s_i \cc_i}.
\]
\end{definition}

\begin{remark}
Using the inclusion-exclusion principle,
if $\mathcal{A},\mathcal{A}_1,...,\mathcal{A}_s$ are subsets
of the set of limit cycles of $h$ and $\mathcal{A}=\bigcup_{i}^{s}\mathcal{A}_i$, then,
\[
\|\mathcal{A}\|=\sum_{J \subseteq [s]}(-1)^{|J|+1}{\|\bigcap_{j \in J} \mathcal{A}_j \|}.
\]
\end{remark}

\begin{definition}
For any subset $J \subseteq [t]$, let
\[
J^{\preceq} := \{k \, : \,  G_j \preceq G_k \mbox{ for some } j \in J\} , \,
J^{\succeq} := \{k \, : \,  G_j \succeq G_k \mbox{ for some } j \in J\},
\]
\[
J^{\prec} := \{k \, : \,G_j \prec G_k \mbox{ for some } j \in J\} \mbox{, and }
J^{\succ} := \{k \, : \,G_j \succ G_k \mbox{ for some } j \in J\}.
\]
A limit cycle $\cc$ in the phase space of $f$ is $J_0$ (resp.
$J_1$) if the $G_j$ component of $\cc$ is $0$ (resp. $1$) for all
$j \in J$.
\end{definition}
Notice that the sets defined above are closely related to upper
and lower order ideals in posets, see \cite[p. 100]{stanley}.

%any limit cycle is $\emptyset_0$ and $\emptyset_1$.

%Example Cont.
\vspace{.2cm} \noindent {\bf \emph{Example}} \ref{ex:run} (Cont.).
Let $J =\{2\}$.  Then
$J^{\preceq} = \{2,3\}, \,
J^{\succeq} = \{1, 2\}, \,
J^{\prec} = \{3\}, \,
\mbox{ and }
J^{\succ} = \{1\}$.

\begin{definition}
A limit cycle $\cc$ in $\s(h)$ is \emph{$J-$regular} if $\cc$ is both
$(J^{\prec})_0$ and $(J^{\succ})_1$, for some maximal antichain
$J$ in $\mathcal{P}$. A limit cycle is called \emph{regular} if it is
$J-$regular for some maximal antichain $J$. That is, an element of
a $J-$regular limit cycle has 0 in all entries corresponding to
strongly connected components lying above components in $J$ and 1
in all entries corresponding to components lying below components
in $J$.
\end{definition}

\begin{lemma}\label{lem:chj}
Let $\jj \subseteq \Omega$ be a set of maximal antichains.
For $J \in \jj$, let $A_J$ be the set of all $J-$regular
limit cycles in the phase space of $f$. Then,
\[
\| \bigcap_{J \in \jj}A_J \|=\prod_{j\in \bigcap_{J \in \jj} J }{\cc(h_j)}.
\]
\end{lemma}
\begin{proof}
It is easy to see that a regular limit cycle $\cc \in \bigcap_{J \in \jj}A_J$
if and only if the only non trivial components of $\cc$ are in $j\in \bigcap_{J \in \jj} J$.
\end{proof}

%%%%%%%%%%%%%%%%%%%%%%%%%%%%%%%%%%%%%%%%%%%%%%%%%%%%%%%%%%%%%%%%%
\section{A Lower Bound}

It is easy to see that a regular cycle is actually a limit cycle
in the phase space of $f$, and hence the numbers of regular cycles
of different length provide lower bounds for the number of limit
cycles of different lengths in the phase space of $f$. Next we
express these numbers in terms of the cycle structure $\cc(h_i)$ of
the strongly connected components.

\begin{lemma}
The cycle structure of the regular limit cycles in $\s(h)$ is
\begin{equation}
\sum_{\jj \subseteq \Omega} (-1)^{|\jj|+1} \prod_{j\in \bigcap_{J \in \jj} J}{\cc(h_j)}. \label{lower-poly}
\end{equation}
\end{lemma}

\begin{proof}
By definition, the set of regular limit cycles is
$
\bigcup_{J\in\Omega}{A_J}.
$
Now, using the inclusion-exclusion principle and Lemma \ref{lem:chj},
it follows that
\begin{align*}
\bigcup_{J\in\Omega}{A_J} &=
\sum_{\mathcal{J}\subseteq \Omega}(-1)^{|\jj|+1} {\| \bigcap_{J \in \jj}A_J \|} \\
 &= \sum_{\jj \subseteq \Omega}
(-1)^{|\jj|+1} \prod_{j\in \bigcap_{J \in \jj} J}{\cc(h_j)}.
\end{align*}
\end{proof}

\begin{theorem}
Consider the function
\begin{equation*}
\mathcal{L}(x_1,\dots,x_t) := \sum_{\jj
\subseteq \Omega} (-1)^{|\jj|+1} \prod_{j\in \bigcap_{J \in \jj}
J} x_i.
\end{equation*}
 Then, for any conjunctive Boolean network $f$ with
subnetworks $h_1,\dots, h_t$ and $\Omega$ it's set of maximal
antichains in the poset of $f$, we have
\begin{equation}
\mathcal{L}(\cc(h_1),\dots,\cc(h_t)) \leq \cc(f). \label{lower-ineq}
\end{equation}
Here, the function $\mathcal{L}$ is evaluated using the
``multiplication" described in Corollary \ref{cross-prod-cycles}.
This inequality provides a sharp lower bound on the number of
limit cycles of $f$ of a given length.
\end{theorem}

\begin{proof}
The inequality follows from the previous lemma.
Now to show that this lower bound is sharp, it is enough to present
a Boolean monomial dynamical system $f$ such that
$\cc(f) = \mathcal{L}(\cc(h_1),\dots,\cc(h_t))$.

Let $f$ be such that whenever there is a directed edge from $G_i$ to $G_j$
in the poset $\pp$,
there is a directed edge from every vertex in $G_i$ to every vertex in $G_j$
in the dependency graph of $f$. It is enough to check that any cycle in the phase
space of $f$ is regular. Let $\cc$ be a cycle in the phase space of $f$. Let
$\cc^j$ be the projection of $\cc$ on the strongly connected
component $G_j$. It is clear that $\cc^j$ is cycle in the $\cc(h_j)$.
Suppose that $\cc^j$ is non trivial. Now if $G_i \preceq G_j$, then
all entries corresponding to $G_i$ must be one.
On the other hand, if $G_i \succeq G_j$, then
all entries corresponding to $G_i$ must be zero.
Therefore, $\cc$ is regular.
\end{proof}

Note that the left side of the inequality (\ref{lower-ineq}) is a
polynomial function in the variables $\cc(h_i)$, with coefficients
that are functions of the cycle numbers of various lengths of the
$h_i$ and the anti-chains of the poset of strongly connected
components.  That is, the sharp lower bound is a polynomial
function depending exclusively on measures of the network
topology.
%%%%%%%%%%%%%%%%%%%%%%%%%%%%%%%%%%%%%%%%%%%%%%%%%%%%%%%%%%%%%%
\section{An Upper Bound} \label{upper-bound-sec}

Next we present a polynomial whose coefficients provide upper
bounds for the number of possible limit cycles in the phase space
of $f$. Similar to the lower bound polynomial, this upper bound
polynomial is in terms of the strongly connected components of $f$
and their poset. However, this upper bound is not sharp. In fact,
we will show that it is not possible to give a sharp upper bound
that is in a polynomial form with constant coefficients. First we prove
the following set-theoretic equality which we need in this
section.
\begin{lemma}\label{lemma_sets}
Let $\{A_i \, : \, i \in \Delta\}$ be a finite collection of finite sets of limit cycles in $\s(h)$, $A=\bigcap_{i}{A_i}$,
and $B_i \subseteq A_i$. Then, for $K \subseteq \Delta$,
\[
\| \bigcap_{k\in K}{(A_k\setminus B_k)} \| =\sum_{L\subseteq
K}{(-1)^{\mid L\mid} \| A\cap \bigcap_{k\in L}{B_k}\|}.
\]
\end{lemma}

\begin{proof}
To begin with, observe that
\begin{align*}
\bigcap_{k\in K}{(A_k\setminus B_k)}  &=\bigcap_{k\in K}{(A\setminus B_k)}\\
&= A\setminus \bigcup_{k\in K}{(A \cap B_k)}.
\end{align*}
Then, using the inclusion-exclusion principle we obtain:
\begin{align*}
\|\bigcap_{k\in K}{(A_k\setminus B_k)}\|
&= \| A\setminus \bigcup_{k\in K}{(A \cap B_k)}\|\\
&=\sum_{L\subseteq K}{(-1)^{\mid L\mid}\| \bigcap_{k\in L}{(A\cap B_k)}\|}\\
&=\sum_{L\subseteq K}{(-1)^{\mid L\mid}\| A\cap \bigcap_{k\in L}{B_k}\|}.
\end{align*}
\end{proof}

\begin{definition}
A limit cycle $\cc$ in $\s(h)$ is called \emph{admissible} if, for
any $j \in [t]$,
\begin{enumerate}
\item If the coordinates of states in $\cc$ corresponding to
component $G_j$ are $0$, then $\cc$ is $(j^{\preceq})_0$; and
\item If the coordinates of states in $\cc$ corresponding to
component $G_j$ are $1$, then $\cc$ is $(j^{\succeq})_1$.
\end{enumerate}
\end{definition}

Notice that any limit cycle in the phase space of $f$ is admissible.
In particular, we have the following lemma.

\begin{lemma}
The number of admissible limit cycles in $\s(h)$ is an upper bound
for the number of limit cycles in $\s(f)$.
\end{lemma}

For $J \subseteq [t]$, let $Z_{J}$ be the set of all $J_0$ limit
cycles and let $O_{J}$ be the set of all $J_1$ limit cycles.
Let $K,L \subseteq [t]$, then it is easy to see that$Z_K\cap Z_L=Z_{K\cup L}$,
$O_K\cap O_L=O_{K\cup L}$, and the cycle structure of $Z_K\cap O_L$ is
\[
\|Z_K\cap O_L\|= \langle K,L \rangle \prod_{j\in \overline{K\cup
L}}{\cc(h_j)},
\]
where
\[
\langle K,L\rangle=
\begin{cases}
0, & \text{if $K \cap L \neq \emptyset$},\\
1, & \text{if $K \cap L = \emptyset $}.
\end{cases}
\]

\begin{theorem}
The cycle structure of the admissible limit cycles in $\s(h)$ is
equal to
\[
\sum_{\substack{I\subseteq N\subseteq [t] \\
J\subseteq M\subseteq [t]}} {(-1)^{|N|+|M|+|I|+|J|}
 {\langle I_N,J^M\rangle \prod_{k\in
\overline{I_N\cup J^M}}{\cc(h_k)} }},
\]
where $I_N=I^{\succeq}\cup N,J^M=J^{\preceq}\cup M $.
\end{theorem}

\begin{proof}
By definition, it follows that the set of admissible limit cycles is

\begin{equation*}
Adm :=\{\textrm{limit cycles in $\s(h)$}\} \backslash \bigcup_{i,j
\in [t]}{ (Z_{\{i\}} \backslash Z_{ \{i\}^{\succeq} }) \cup
(O_{\{j\}} \backslash O_{ \{j\}^{\preceq} })  }.
\end{equation*}

Using the inclusion-exclusion principle, it follows that

\begin{equation}\label{inc_exc}
\|Adm\| =
\prod_{k\in [t]}{\cc(h_k)}+\sum_{ \substack{ N\subseteq [t] \\
M\subseteq [t]\\N\cup M\neq \emptyset}} { (-1)^{|N|+|M|}
\|\bigcap_{i\in N}{ (Z_{\{i\}} \backslash
Z_{\{i\}^{\succeq} }) } \cap \bigcap_{j\in M} {(O_{\{j\}}
\backslash O_{ \{j\}^{\preceq} }) } \| }.
\end{equation}

On the other hand, since $\bigcap_{i\in N}{ Z_{\{i\}}}=Z_N$, \
$\bigcap_{j\in M}{ O_{\{j\}}}=O_M$, \ $Z_{i}\subseteq
Z_{\{i\}^{\succeq}}$ and $O_{j}\subseteq O_{\{j\}^{\preceq}}$ it
follows from Lemma \ref{lemma_sets} that

\begin{eqnarray*}
\|  \bigcap_{i\in N}{ (Z_{\{i\}} \backslash Z_{\{i\}^{\succeq} }) }
&\cap&  \bigcap_{j\in M} {O_{\{j\}}\backslash O_{ \{j\}^{\preceq} }) }\| \\
 &=& \sum_{\substack{I\subseteq N\\J \subseteq M}}
   {(-1)^{|I|+|J|}\| (Z_N\cap O_M)\cap
(   \bigcap_{i\in N}{Z_{\{i\}^{\succeq}}}\cap
    \bigcap_{j\in M}{O_{\{j\}^{\preceq}}})
    \| } \\
&=& \sum_{\substack{I\subseteq N\\J \subseteq M}}
   {(-1)^{|I|+|J|}\| (Z_N\cap O_M)\cap
(   Z_{\bigcup_{i\in N}\{i\}^{\succeq}} \cap
    O_{\bigcup_{j\in M}\{j\}^{\preceq}})
    \| } \\
&=& \sum_{\substack{I\subseteq N\\J \subseteq M}}{(-1)^{|I|+|J|} \|
Z_N \cap O_M\cap Z_{I^{\succeq}} \cap O_{J^{\preceq}} \| } \notag  \\
&=& \sum_{\substack{I\subseteq N\\J \subseteq M}}{(-1)^{|I|+|J|} \|
Z_{N \cup I^{\succeq}} \cap O_{M \cup J^{\preceq}} \| } \\
&=&  \sum_{\substack{I\subseteq N\\J \subseteq M}}{(-1)^{|I|+|J|} \|
Z_{I_N} \cap O_{J^M} \| } \\
&=& \sum_{\substack{I\subseteq N\\J
\subseteq M}}{(-1)^{|I|+|J|} {\langle I_N,J^M\rangle \prod_{k\in
\overline{I_N\cup J^M}}{\cc(h_k)} } }.
\end{eqnarray*}
Equation (\ref{inc_exc}) then becomes
\begin{align*}
\|Adm\|&=\prod_{k\in [t]}{\cc(h_k)}+
\sum_{ \substack{ N\subseteq [t] \\
M\subseteq [t]\\N\cup M\neq \oslash}} { (-1)^{|N|+|M|}
\sum_{\substack{I\subseteq N\\J \subseteq M}}{(-1)^{|I|+|J|}
{\langle I_N,J^M\rangle \prod_{k\in \overline{I_N\cup
J^M}}{\cc(h_k)} }}
 }\\
 &=\prod_{k\in [t]}{\cc(h_k)}+
\sum_{\substack{ I\subseteq N\subseteq [t] \\
J\subseteq M\subseteq [t]\\N\cup M\neq
\oslash}}{(-1)^{|N|+|M|+|I|+|J|} {\langle I_N,J^M\rangle
\prod_{k\in \overline{I_N\cup J^M}}{\cc(h_k)} }}
 \\
&=\sum_{\substack{I\subseteq N\subseteq [t] \\
J\subseteq M\subseteq [t]}} {(-1)^{|N|+|M|+|I|+|J|}
 {\langle I_N,J^M\rangle \prod_{k\in
\overline{I_N\cup J^M}}{\cc(h_k)} }}.
\end{align*}

\end{proof}

\begin{corollary}
For $K,L \subseteq [t]$, let
\[
\phi_{(K,L )}(x_1,\dots,x_t) = \left\{%
            \begin{array}{ll}
            0, & \mbox{ if } K \cap L \neq \emptyset, \\
             x_{\overline{K\cup L}}, & \mbox{ if } K \cap L = \emptyset. \\
            \end{array}%
            \right.
\]
Consider the function
\begin{equation}\label{upper-poly}
\mathcal{U}(x_1,\dots,x_t) = \sum_{I \subseteq N, \, J \subseteq
M} (-1)^{|N|+|M|+|I|+|J|} \phi_{(I_N,J^M)}(x_1,\dots,x_t)
\end{equation}
then
\begin{equation}
\cc(f)\leq \mathcal{U}(\cc(h_1),\dots,\cc(h_t)),\label{upper-ineq}
\end{equation}
which provides therefore an upper bound for the cycle structure of
$f$.
\end{corollary}
This upper bound is clearly not sharp, since it is easy to find admissible limit
cycles in the phase space of $h$ that are not cycles in the phase space of $f$.
The advantage of this upper bound, however, is that it is a polynomial in terms of
the strongly connected components and their poset. Furthermore, we will show next
that this bound gives the exact number of fixed points.

\begin{theorem}
Let $f$ be a conjunctive Boolean network. Then any admissible fixed point is regular
and hence the number of fixed points $C(f)_0$ in the phase space of $f$ is
\begin{equation}
% \nonumber to remove numbering (before each equation)
  C(f)_0 = \sum_{\jj \subseteq \Omega} (-1)^{|\jj|+1} \, 2^{|\bigcap_{J \in \jj} J|}.  \label{fp-exact}
\end{equation}
\end{theorem}
\begin{proof}
It follows from the definitions that any admissible fixed point is regular, and hence
$\mathcal{U}(2\cc_1,\dots,2\cc_1) = \mathcal{L}(2\cc_1,\dots,2\cc_1)$. In particular, the
number of fixed points $C(f)_0$ is the coefficient of $\cc_1$ in  $\mathcal{L}(2\cc_1,\dots,2\cc_1)$.
By Equation (\ref{lower-poly}), we get
\begin{eqnarray*}
  \mathcal{L}(2\cc_1,\dots,2\cc_1)
   &=& \sum_{\jj \subseteq \Omega} (-1)^{|\jj|+1} \prod_{j\in \bigcap_{J \in \jj} J} 2\cc_1 \\
   &=& \sum_{\jj \subseteq \Omega} (-1)^{|\jj|+1} \,  2^{|\bigcap_{J \in \jj} J|} \cc_1.
\end{eqnarray*}
\end{proof}

\subsection{An example}

Each system $g$ from Corollary \ref{upperbound-g} has
the same poset as $f$ and the cycle structure of each $g$ is an
upper bound for the cycle structure of $f$. However, using the
poset alone, or using the cycle structure of the strongly
connected components alone, one can not expect to find a
polynomial form with constant coefficients that would provide a
sharp upper bound, as we now show.

Let $\pp$ be the poset on two nodes $G_1$ and $G_2$ where $G_1
\preceq G_2$. Let $f$ be a conjunctive Boolean network with a
dependency graph that has only two strongly connected components
which are connected by one edge. Let $F$ be the set of all such
systems. We will show that there is no polynomial $W(x_1,x_2) =
\sum_{i,j} a_{ij}x_1^ix_2^j$ such that, for all $f \in F$,
 $W(\cc(h_1),\cc(h_2)) = \cc(f)$, where $h_1$ and $h_2$ are the
 conjunctive Boolean networks
corresponding to the two strongly connected components of $f$.

Suppose the loop number of $h_1$ is 1 and that of $h_2$ is $q$
where $q$ is prime. Then $\cc(h_1) = 2\cc_1$ and $\cc(h_2) =
2\cc_1 + \dfrac{2^q-2}{q}\cc_q$. It is easy to check that $\cc(f)
= 3\cc_1 +\dfrac{2^q-2}{q}\cc_q$. Since $\cc(f) =
W(\cc(h_1),\cc(h_2))$, we have

\begin{eqnarray*}
\frac{2^q-2}{q}\cc_q+3\cc_1 &=& W(2\cc_1,2\cc_1+\frac{2^q-2}{q}\cc_q) \\
&=& \sum_{i,j}{a_{ij}(2\cc_1)^i(2\cc_1+\frac{2^q-2}{q}\cc_q)^j} \\
&=& \sum_{i,j} 2^i a_{ij} (2\cc_1+\frac{2^q-2}{q}\cc_q)^j \\
&=& \sum_{i,j} 2^i a_{ij} [2^j\cc_1 + [(2+\frac{2^q-2}{q})^j-2^j]\cc_q] \\
&=& \sum_{i,j} 2^{i+j} a_{ij} \cc_1 + \sum_{i,j} 2^i
a_{ij}[(2+\frac{2^q-2}{q})^j-2^j]\cc_q.
\end{eqnarray*}

Equating coefficients,  for any $q$, we have
$\frac{2^q-2}{q} = \sum_{i,j} 2^i a_{ij}[(2+\frac{2^q-2}{q})^j-2^j]$.
Therefore, $a_{ij} =0$ for all $i,j \geq 1$ and hence $W(x_1,x_2)$ must be
of the form $W(x_1,x_2) = a_{11}x_1x_2+ a_{10}x_1+a_{01} x_2- a_{00}$.

Now suppose the loop number of $h_1$ is $p$ and that of $h_2$ is
$q$. Then any cycle in the phase space of $f$ is regular and hence
it is easy to check that $\cc(f) = \cc(h_1)+\cc(h_2)-\cc_1$. In
particular, $a_{11}=0$ and hence $W(x_1,x_2) = a_{10}x_1+a_{01}
x_2- a_{00}$.

However, for the case when the loop number of $h_1$ is $4$ and that of $h_2$ is $6$,
we have $\cc(f) = 3\cc_1+3\cc_2+2\cc_3+ 4\cc_4+11\cc_6+2\cc_{12}$, \,
$\cc(h_1) = 2\cc_1+\cc_2+3\cc_4$ and $\cc(h_2) = 2\cc_1+\cc_2+2\cc_3+9\cc_6$.
In particular, $\cc(f) \neq W(\cc(h_1),\cc(h_2))$.

%%%%%%%%%%%%%%%%%%%%%%%%%%%%%%%%%%%%%%%%%%%%%%%%%%%%%%%%%%%%%%%%%%%%%%%%%%
\section{Discussion}
In this paper we have focused on the class of conjunctive  and disjunctive Boolean
networks and have treated the problem of predicting network
dynamics from the network topology.  Such networks are entirely determined
by the topology of their dependency graphs, so it should be
possible in principle to extract all information about the
dynamics from this topology. This problem has been the subject of many
research articles and been solved for the class of XOR Boolean networks, that is,
networks where all nodes use the XOR Boolean operator.
Determining the dynamics of a network from its topology has
practical importance, for instance in the study of
dynamical processes on social networks, such as the spread of an
infectious disease. Public health interventions, such as
quarantine of selected individuals or closure of certain
institutions or means of travel, often attempt to alter network
dynamics by altering network topology.  The current study can be
seen as a theoretical study that begins to elucidate the role of
certain features of the network topology in supporting certain
types of dynamics. Another example is discussed in \cite{AO},
where it is shown that the dynamics of a certain gene regulatory
network in the fruit fly \emph{Drosophila melanogaster} is
determined by the topology of the network.

We have shown that if the dependency graph of the network is
strongly connected, then the key topological determinant of
dynamics is the loop number, or index of imprimitivity, of the
graph, and one can describe the cycle structure exactly. The
approach here is to look at powers of the incidence matrix of the
graph in the max-plus algebra.  In light of Theorem
\ref{Th:strongly_connected} one could argue that this is the
prevalent case for large networks. What is more, this theorem
shows that large networks tend to have loop number one, so that
their dynamics is extraordinarily simple, consisting only of two
fixed points.

For general dependency graphs, we give lower and upper bounds on
the cycle structure in the form of a polynomial function
determined by structural measures of the graph.  The lower bound
is sharp.

In some sense, the main contribution of this paper lies in the
identification of a class of networks that provide an excellent
case study of the relationship between network topology and
dynamics, which can provide hypotheses for more general classes of
networks.

%%%%%%%%%%%%%%%%%%%%%%%%%%%%%%%%%%%%%%%%%%%%%%%%%%%%%
%   Bibliography

\end{document}